\documentclass[11pt]{amsart}
\usepackage[usenames]{color}
\usepackage{fullpage}
\usepackage{amscd}
\usepackage{amssymb, latexsym}
\usepackage[all,2cell,ps]{xy}
\usepackage{mathdots}

\theoremstyle{plain}
\newtheorem{thm}{Theorem}[section]
\newtheorem{theorem}[thm]{Theorem}

\newtheorem{corollary}[thm]{Corollary}

\newtheorem{lemma}[thm]{Lemma}

\newtheorem{proposition}[thm]{Proposition}

\newtheorem{fact}[thm]{Fact}

\theoremstyle{definition}
\newtheorem{de}[thm]{Definition}
\newtheorem{exm}[thm]{Example}

\newtheorem{example}[thm]{Example}

\newcommand{\cz}[1]{{\color{red}{#1}\color{black}{}}}

\newcommand{\id}{\mathrm{id}}

\newcommand{\Mlt}{\mathop{\mathcal{G}}}

\numberwithin{equation}{section}

\begin{document}

\title{Diagonals of solutions of the Yang-Baxter equation}

\author{P\v remysl Jedli\v cka}
\author{Agata Pilitowska}

\address{(P.J.) Department of Mathematics and Physics, Faculty of Engineering, Czech University of Life Sciences, Kam\'yck\'a 129, 16521 Praha 6, Czech Republic}
\address{(A.P.) Faculty of Mathematics and Information Science, Warsaw University of Technology, Koszykowa 75, 00-662 Warsaw, Poland}

\email{(P.J.) jedlickap@tf.czu.cz}
\email{(A.P.) agata.pilitowska@pw.edu.pl}

\keywords{Yang-Baxter equation, diagonal mappings, multipermutation solutions, reductive solutions}

\subjclass[2020]{Primary: 16T25. 
Secondary: 08A05, 20B30.
}

\date{\today}

\begin{abstract}
We study the diagonal mappings in non-involutive set-theoretic solutions of the Yang-Baxter equation. We show that, for non-degenerate solutions, they are commuting bijections. This gives the positive answer to the question: ``Is every non-degenerate solution bijective?'' of Ced\'{o}, Jespers and Verwimp. 
Additionally, we show that for a subclass of solutions called $k$-\emph{permutational}, only one-sided non-degeneracy suffices to prove that one of the diagonal mappings is invertible. We also present an equational characterization of multipermutation solutions and extend results of Rump and Gateva-Ivanova about decomposability to non-involutive case. In particular, we show that each, not necessarily involutive, square-free multipermutation solution of finite level and arbitrary cardinality, is always decomposable. 
\end{abstract}

\maketitle

\section{Introduction}

The Yang-Baxter equation is a fundamental equation occurring in mathematical physics. It appears, for example, in integrable models in statistical mechanics, quantum field theory or Hopf algebras~(see e.g. \cite{Jimbo, K}). Searching for its solutions has been absorbing researchers for many years.

Let us recall that, for a vector space $V$, a {\em solution of the Yang--Baxter equation} is a linear mapping $r:V\otimes V\to V\otimes V$ 
 such that
\begin{align*}
(id\otimes r) (r\otimes id) (id\otimes r)=(r\otimes id) (id\otimes r) (r\otimes id).
\end{align*}

Description of all possible solutions seems to be extremely difficult and therefore
there were some simplifications introduced by Drinfeld in \cite{Dr90}.
Let  $X$ be a basis of the space $V$ and let $\sigma:X^2\to X$ and $\tau: X^2\to X$ be two mappings. We say that $(X,\sigma,\tau)$ is a {\em set-theoretic solution of the Yang--Baxter equation} if
the mapping 
$$x\otimes y \mapsto \sigma(x,y)\otimes \tau(x,y)$$ extends to a solution of the Yang--Baxter
equation. It means that $r\colon X^2\to X^2$, where $r=(\sigma,\tau)$,  
satisfies the \emph{braid relation}:
\begin{equation}\label{eq:braid}
(id\times r)(r\times id)(id\times r)=(r\times id)(id\times r)(r\times id).
\end{equation}

A solution is called {\em left non-degenerate} if the mappings $\sigma_x=\sigma(x,\_)$  are bijections
	and {\em right non-degenerate} if the mappings $\tau_y=\tau(\_\,,y)$
		are bijections,
for all $x,y\in X$. A solution is {\em non-degenerate} if it is left and right non-degenerate.
A~solution is called {\em bijective} if $r$ is a~bijection of~$X^2$. It is {\em involutive} if $r^{-1}=r$, i.e., for each $x,y\in X$, $\tau_y(x)=\sigma_{\sigma_x(y)}^{-1}(x)$ and $\sigma _x(y)=\tau^{-1}_{\tau_y(x)}(y)$. Moreover, it is
\emph{square-free} if $r(x,x)=(x,x)$, for every $x\in X$.

All solutions we study in this paper are set-theoretic but not necessarily non-degenerate and we will call them simply \emph{solutions}. 
The set $X$ can be of arbitrary cardinality.

If $(X, \sigma,\tau)$ is a solution then directly, by the braid relation, we obtain, for $x,y,z\in X$:
\begin{align}
&\sigma_x\sigma_y=\sigma_{\sigma_x(y)}\sigma_{\tau_y(x)} \label{birack:1}\\
&\tau_{\sigma_{\tau_y(x)}(z)}\sigma_x(y)=\sigma_{\tau_{\sigma_y(z)}(x)}\tau_{z}(y) \label{birack:2}\\
&\tau_x\tau_y=\tau_{\tau_x(y)}\tau_{\sigma_y(x)} \label{birack:3}
\end{align}

\vskip 4mm

In the last decade, the main interest of researchers lied in the investigation of involutive solutions.
Non-involutive solutions were firstly studied by Soloviev \cite{Sol} and Lu, Yan and Zhu \cite{LYZ} and later by Guarnieri, Vendramin \cite{GV}, Colazzo, Jespers, Van Antwerpen, Verwimp \cite{CJAV} and Rump \cite{Rump19} and many others, mostly in connection with so called {\em skew braces}, YB-\emph{semitrusses} and $q$-\emph{cycle sets}. A special emphasis was taken on two classes of solutions: {\em indecomposable} and {\em multipermutation} ones. In this context researchers mainly concentrated on involutive solutions. Following Vendramin~\cite{V24}: {\em ``It is worth
noting that our knowledge about these families in the non-involutive case is
currently limited, and there is still much to be explored and understood.''}  we try to expand our knowledge about non-involutive solutions. 

Etingof,  Schedler and Soloviev observed in \cite[Proposition 2,2]{ESS} (see also \cite{JPZ19}) that, for a non-degenerate involutive solution $(X,\sigma,\tau)$, the mappings 
\begin{align*}
&U\colon X\to X; \quad U(x)= \sigma^{-1}_x(x),
\quad {\rm and} \\
&T\colon X\to X; \quad T(x)= \tau^{-1}_x(x)
\end{align*} are invertible, $U=T^{-1}$   and
\begin{align}\label{eq:TU}
&r(T(x),x)=(T(x),x)\quad {\rm and}\quad  r(x,U(x))=(x,U(x)).
\end{align}
Such correspondence is true not only for involutive solutions. By results of Smoktunowicz and Vendramin \cite[Corollary 3.3]{SV18} for solutions which originate from so called \emph{skew braces} also $U=T^{-1}$. 
Additionally, by observation of Lebed and Vendramin \cite[Lemma 1.4]{LV}, the relationship \eqref{eq:TU} 
is satisfied in any finite invertible, non-degenerate, \emph{injective} solution too. But, by Examples \ref{ex:Lub} and \ref{ex:der}, it does not hold in general. On the other hand, Stanovsk\'y proved \cite{S06} that the identities \eqref{eq:TU} are equivalent in any non-degenerate solution.

Moreover, it is evident that a non-degenerate solution $(X,\sigma,\tau)$ is square-free if and only if $T=U=\id$. 
The permutation $T$ helps to decompose the set $X$ into a disjoint union of subsolutions and plays an important role as a criterion of decomposability of a solution.

The investigation of indecomposable solutions was initiated  by Etingof et. al. in \cite{ESS}. In 1996 at the International Algebra Conference in Miskolc Gateva-Ivanova stated a conjecture that each finite involutive square-free solution is decomposable. In 2005 Rump \cite{Rump05} proved the conjecture and additionally showed that it does not hold for infinite solutions.  Ramirez and Vendramin  \cite{RV} noticed that the diagonal mappings may carry deep information about the structure of solutions and Rump's Theorem may be reformulated in the following way: \emph{Each finite non-degenerate involutive solution, for which the diagonal map fixes all points of $X$, is decomposable}. They modified the assumption of Rump's Theorem and considered solutions with special structure of the bijection $T$. In particular, they showed that non-degenerate involutive solutions of size $n>1$, for which $T$ is a cycle of length $n$, are indecomposable. But if the diagonal is a cycle of length $n-1$ then such solutions are decomposable.

Later on, Camp-Mora and Sastriques obtained in \cite{CS} that a finite non-degenerate involutive solution $(X,\sigma,\tau)$, such that the order of the permutation $T$ and the cardinality $|X|>1$ are coprime, is decomposable. Just recently, Lebed, Ramirez and Vendramin \cite{LRV} obtained new results devoted to decomposability of non-degenerate involutive solutions regarding different conditions on the diagonal mappings. The importance of the diagonal mappings is also visible in \cite{AMV}, where such bijections help to enumerate non-degenertate involutive solutions of small size. 

In \cite[Question 4.3]{CJV} Ced\'{o}, Jespers and Verwimp stated the question whether each non-degenerate solution is always bijective and they showed \cite[Theorem 4.5]{CJV} that it is true if the solution is irretractable. Later on, Colazzo, Jespers, Van Antwerpen and Verwimp extended the result a little~\cite[Proposition 3.6]{CJAV}. Moreover, they showed \cite[Theorem 3.1]{CJAV} (see also \cite[Corollary 6]{CCS}) that, for a finite left non-degenerate solution, the solution is bijective if and only if it is right non-degenerate. It turned out that the bijectivity of the diagonal mappings plays crucial role. In particular,  Ced\'{o}, Jespers and Verwimp proved \cite[Lemma 4.4]{CJV} that, for a non-degenerate irrectractable solution, the diagonal $U$ is a bijection and the inverse is equal to $T$. On the other hand, Rump's results \cite[Corollary 2 of Proposition 8]{Rump19} imply that, for left non-degenerate bijective solutions, the diagonal $U$ is invertible if and only if the solution is right non-degenerate. Further, Colazzo, Jespers, Van Antwerpen and Verwimp showed \cite[Lemma 3.2, Lemma 3.3]{CJAV} that, for non-degenerate solutions, the diagonal operations $U$ and $T$ are always injective and they are bijections if and only if the solution is bijective. 

In our paper we give a straightforward proof that the diagonal mappings are bijections in any non-degenerate solution. This immediately gives that any non-degenerate solution is bijective and affirmatively answers Question 4.3 \cite{CJV} of Ced\'{o}, Jespers and Verwimp.   
Moreover, we prove that, although the diagonals $U$ and $T$   are not mutually inverse in general, they commute.
Additionally, we show that, for a subclass of solutions called \emph{permutational}, one-sided non-degeneracy is sufficent for the diagonal to be invertible.

Another topic that we study are multipermutation solutions. They receive an attention
since the property resembles nilpotency and also has strong impact on decomposability of a solution. 
Gateva-Ivanova studied such
solutions in many papers. In \cite[Theorem 4.14]{GIC12} together with Cameron they showed that square-free involutive solutions (of arbitrary cardinality) with finite multipermutation level are  decomposable. Additionally,
Ced\'{o} and Okni\'{n}ski proved in \cite{CO21} that indecomposable involutive solutions of square-free cardinality are always multipermutation solutions. 

In this paper we 
give an equational characterization of multipermutation solutions
and we apply this characterization to
extend results of Rump and Gateva-Ivanova about decomposability to non-involutive solutions. Recently, similar results were presented by Castelli and Trappeniers~\cite{CT}, using a different technique.

The paper is organized as follows: 
in Section~2 we give the definition of a multipermutation solution and compare different approaches.
In Section~\ref{sec:diag} we prove that the diagonal mappings $U$ and $T$ are commutative bijections in each non-degenerate solution  
and we present the inverses to the diagonal mappings [Theorem \ref{lm:TSbijections}]. Some technical results were obtained using the automated deduction software Prover9~\cite{Prover}. In~Section~\ref{sec:multi}
we introduce equational characterization of multipermutation solutions [Proposition \ref{prop:per}]. For multipermutation solutions we give a simpler form of inverses to the diagonal mappings. In particular, we show that one-sided non-degeneracy is sufficent for the diagonal to be invertible [Proposition \ref{lm:bijections}]. Finally, in Section \ref{sec:kred} we define $k$-reductive solutions. We extend Rump's and Gateva-Ivanova's Theorems to non-involutive non-degenerate solutions and  we show that each non-degenerate non-involutive square-free solution of multipermutation level $k$ and arbitrary cardinality is  decomposable [Corollary \ref{cor:sqf-dec}].
At the end we discuss the modified  definition of the multipermutation level of a solution introduced by Castelli and Trappeniers  in \cite{CT} for $k$-reductive solutions.

\section{Retracts}

In this section we discuss different concepts of rectracts that appeared in the literature. All the proofs in this section are probably known to some extent however we felt that a formal proof should
appear somewhere and this is why we decided to write them down.

Let $X$ be a set. An equivalence relation $\mathord{\asymp}\subseteq X\times X$ is \emph{compatible} with a binary operation $f\in X^X$ if for $x_1,x_2,y_1,y_2\in X$,
\begin{equation}
	f(x_1,y_1)\asymp f(x_2,y_2),
\end{equation}
whenever $x_1\asymp x_2$ and $y_1\asymp y_2$.
This condition ensures that there exists a well defined operation
$f/_{\asymp}$ on the equivalence classes of~$\asymp$.
If an equivalence is compatible with all the operations of a structure then it is called a {\em congruence}.

In \cite{ESS} Etingof, Schedler and Soloviev introduced, for each involutive solution $(X,\sigma,\tau)$, the equivalence relation $\sim$ on the set $X$: for each $x,y\in X$
\begin{align}\label{rel:sim}
	x\sim y\quad \Leftrightarrow\quad \sigma_x=\sigma_y
\end{align}
and they showed that 
$\sim$ is a compatible with the operations $\sigma$, $\sigma^{-1}$, $\tau$ and $\tau^{-1}$. 
In the case of non-involutive solutions,
the relation~$\sim$ is defined as 
\begin{equation}\label{eq:retract}
	x\sim y \quad\Leftrightarrow \quad {\sigma_x=\sigma_y} \wedge {\tau_x=\tau_y}.
\end{equation}

Lebed and Vendramin showed in \cite{LV} that  the relation $\sim$ is compatible with the operations $\sigma$, $\sigma^{-1}$, $\tau$ and $\tau^{-1}$ in any finite bijective non-degenerate injective solution. 
In \cite{JPZ19} the authors together with Zamojska-Dzienio proved that the relation $\sim$ induces a non-degenerate solution on the quotient set $X^{\sim}$  for any non-degenerate solution $(X,\sigma,\tau)$.  A substantially shorter proof has recently appeared in \cite{CJKAV} for bijective non-degenerate solutions.

\begin{de}\label{ret}
	Let $(X,\sigma,\tau)$ be a non-degenerate solution. The quotient solution $\mathrm{Ret}(X,\sigma,\tau):=(X^{\sim},\sigma,\tau)$  with $\sigma_{x^{\sim}}(y^{\sim})=\sigma_x(y)^{\sim}$ and $\tau_{y^{\sim}}(x^{\sim})=\tau_y(x)^{\sim}$, for $x^{\sim},y^{\sim}\in X^{\sim}$  and $x\in x^{\sim},\; y\in y^{\sim}$, is called the \emph{retraction} solution of $(X,\sigma,\tau)$. One defines \emph{iterated retraction} in the following way: ${\rm Ret}^0(X,\sigma,\tau):=(X,\sigma,\tau)$ and
	${\rm Ret}^k(X,\sigma,\tau):={\rm Ret}({\rm Ret}^{k-1}(X,\sigma,\tau))$, for any natural number $k>1$. 
	We say that a solution $(X,\sigma,\tau)$ is \emph{irretractable} if ${\rm Ret}(X,\sigma,\tau)=(X,\sigma,\tau)$, i.e.
	$\sim$ is the trivial relation.
	On the other hand, for the least 
		integer~$k\geq 0$ such that $\mathrm{Ret}^k(X,\sigma,\tau)$ has one element only  we say that
	$(X,\sigma,\tau)$ has {\em multipermutation level $k$}.
\end{de}

In the case of not non-degenerate solution, an equivalence so defined may be incompatible with some of the operations $\sigma$, $\sigma^{-1}$, $\tau$ and $\tau^{-1}$,  
as we can see on the following example.

\begin{exm}\label{exm:1}
	Let $X=\{a,b,c\}$ and let
	\[ \sigma_x(y)=y, \text{for all }x,y\in X, \qquad 
	\tau_x(y)=\begin{cases} b & \text{if } x\neq c\text{ and }y=a,\\
		c & \text{if } x=c \text{ or }y\neq a.
	\end{cases}	
	\]
	Then $(X,\sigma,\tau)$ is a left non-degenerate solution that is not right non-degenerate, and
	$a\sim b$. However $\tau_{\tau_a(a)}(a)=b\neq c=\tau_{\tau_a(b)}(a)$, which means that $\tau_a(a)\not\sim \tau_a(b)$
  and therefore $\sim$ is not compatible with the operation $\tau$.
 \end{exm}

Nevertheless, there is a class of right degenerate solutions for which a retract equivalence can be defined. In order to work with them, we recall a definition from~\cite{Rump19}.

\begin{de}
A {\em $q$-cycle set} $(X,\cdot,:)$ is a set with two binary operations that
satisfy, for all $x,y,z\in X$,
\begin{align}
	(x\cdot y)\cdot(x\cdot z)&=(y:x)\cdot (y\cdot z)\label{eq:qcs1}\\
	(x\cdot y):(x\cdot z)&=(y:x)\cdot (y: z)\label{eq:qcs2}\\
	(x\cdot y):(x: z)&=(y:x): (y: z)\label{eq:qcs3}
\end{align}
and the mapping $y\mapsto x\cdot y$ is invertible, for all $x\in X$.
The $q$-cycle set is called {\em regular} if the mapping  $y\mapsto x: y$ is also invertible, for all $x\in X$.
\end{de}
\noindent
Note that in fact 
any regular $q$-cycle set has structure of two left-quasigroups. 

The notion of $q$-cycle sets brings a new perspective to solutions
since, according to~\cite{Rump19}, a $q$-cycle set is essentially the same as a left non-degenerate solution and a regular~$q$-cycle set
is essentially the same as a bijective left non-degenerate
solution~\cite[Proposition 1]{Rump19}. In more details, let $(X,\cdot,:)$ be a $q$-cycle set with $\sigma_x(y):=x\cdot y$, for $x,y\in X$. Then $(X,\lambda, \rho)$ for $\lambda_x(y)=\sigma_x^{-1}(y)$ and $\rho_x(y)=\sigma_x^{-1}(y):x$ is a left non-degenerate solution. On the other hand, if $(X,\sigma, \tau)$ is a left non-degenerate solution then $(X,\cdot,:)$ for operations given by: $x\cdot y:=\sigma^{-1}_x(y)$ and $x:y=\tau_{\sigma_y^{-1}(x)}(y)$, is a $q$-cycle set.
This correspondence naturally brings the following definition.

\begin{de}
	A solution $(X,\sigma,\tau)$ is called {\em regular} if it is left non-degenerate and bijective.
\end{de}

If a solution $(X,r)$ where $r=(\sigma,\tau)$, is bijective then
the inverse mapping $r^{-1}$ defines a solution as well
and we denote $r^{-1}=(\hat\sigma,\hat\tau)$. We obtain, from the definition, the following formulas:

\begin{align}
	\sigma_{\hat\sigma_x(y)}\hat\tau_y(x)&=x \label{rr:1}\\
	\tau_{\hat\tau_y(x)}\hat\sigma_x(y)&=y \label{rr:2}\\
	\hat\sigma_{\sigma_x(y)}\tau_y(x)&=x \label{rr:3}\\
	\hat\tau_{\tau_y(x)}\sigma_x(y)&=y \label{rr:4}
\end{align}

\begin{lemma}\label{lm:invreg}
	The inverse solution to a regular solution is regular.
\end{lemma}

\begin{proof}
	For a regular solution $(X,\sigma,\tau)$,
	we will prove that $\hat\sigma_x^{-1}(y)=\tau_{\sigma^{-1}_y(x)}(y)$,
	for all $x,y\in X$:
	\begin{align*}
		\hat\sigma_x(y)\tau_{\sigma^{-1}_y(x)}(y)
   &=\hat\sigma_{\sigma_y\sigma_y^{-1}(x)}\tau_{\sigma^{-1}_y(x)}(y)\stackrel{\eqref{rr:3}}=  y,\text{ and}\\
				\tau_{\sigma^{-1}_{\hat\sigma_x(y)}}(x)\hat\sigma_x(y)&\stackrel{\eqref{rr:1}}=
\tau_{\hat\tau_y(x)}\hat\sigma_x(y)\stackrel{\eqref{rr:2}}=y.\qedhere
	\end{align*}
\end{proof}

From the proof of Lemma \ref{lm:invreg} 
we can see that, for any regular solution,
$x:y=\tau_{\sigma^{-1}_y(x)}(y)=\hat\sigma_x^{-1}(y)$.

\begin{corollary}
    A bijective solution $(X,\sigma,\tau)$ is non-degenerate if and only if the inverse solution $(X,\hat{\sigma},\hat{\tau})$ is non-degenerate.
\end{corollary}

Having a bijective solution we can naturally define the following equivalence relation:
\begin{equation}\label{eq:inv_retract}
	x\mathbin{\hat\sim} y \quad\Leftrightarrow \quad {\hat\sigma_x=\hat\sigma_y} \wedge {\hat\tau_x=\hat\tau_y}.
\end{equation}
We shall prove later that there are classes of solutions where the relations $\sim$ and $\hat\sim$ coincide. Actually, so far we do not have any example where these relations are different. And the same can be said about another formally different relation:

\begin{de}[\cite{CCS}]
	Let $(X,\cdot,:)$ be a $q$-cycle set. We define an equivalence $\sim_:$
	on~$X$ as follows:
	\begin{align}\label{rel:qset}
 &x\sim_: y \qquad \Leftrightarrow \qquad x\cdot z=y\cdot z \text{ and } x:z=y:z, \text{ for all }z\in X.
 \end{align}
In the language of solutions, the definition looks as follows:

Let $(X,\sigma,\tau)$ be a left non-degenerate set. We define an equivalence $\sim_:$
	on~$X$ as follows:
	\begin{align}\label{rel:qset2}
 &x\sim_: y \qquad \Leftrightarrow \qquad \sigma_x^{-1}=\sigma_y^{-1} \text{ and } \tau_{\sigma^{-1}_z(x)}(z)=\tau_{\sigma^{-1}_z(y)}(z), \text{ for all }z\in X.
 \end{align}
\end{de}

Analogously as in the previous case, we may define, for a regular solution, an equivalence using the operations of the inverse solution:
\begin{align}\label{rel:qset3}
 &x\mathbin{\hat{\sim}_:} y \qquad \Leftrightarrow \qquad \hat\sigma_x^{-1}=\hat\sigma_y^{-1} \text{ and } \hat\tau_{\hat\sigma^{-1}_z(x)}(z)=\hat\tau_{\hat\sigma^{-1}_z(y)}(z), \text{ for all }z\in X.
 \end{align}
However, 
in this case, the relation \eqref{rel:qset2} may be express equivalently:
\begin{align}\label{rel:reg}
&x\sim_{:} y \qquad \Leftrightarrow \qquad  \sigma^{-1}_x=\sigma^{-1}_y\qquad \text{and} \qquad \hat\sigma^{-1}_x=\hat\sigma^{-1}_y,
\end{align}
and we see that $\sim_:$ and $\hat\sim_:$ coincide since $\widehat{\hat\sigma}_x=\sigma_x$, for any~$x\in X$.

\begin{proposition}[{\cite[Proposition 10]{CCS}}]\label{cong-qmagma}
	Let $(X,\cdot,:)$ be a regular $q$-cycle set. Then the relation $\sim_:$ is compatible with the operations $\cdot$ and $:$.
\end{proposition}
Proposition \ref{cong-qmagma} 
tells us that the operations $\cdot$ and $:$ are well defined on the equivalence 
classes, however, it tells us nothing whether the operations $\cdot$ and $:$ are still invertible in the quotient.
Thus we do not know whether the quotient is a $q$-cycle set or not, let alone a regular one. 

In the language of solutions, we do not know, whether the equivalence $\sim_:$ is compatible with $\sigma$ and $\hat\sigma$, let alone with $\tau$. We have a proof of the compatibility under some additional assumptions only.

\begin{proposition}
	Let $(X,\sigma,\tau)$ be a regular solution and  assume that, 
 for any $x\in X$, the permutations $\sigma_x$ and $\hat\sigma_x$ are of finite order. Then the relation $\sim_:$ is compatible with the operations $\sigma$, $\sigma^{-1}$, $\hat\sigma$ and $\hat\sigma^{-1}$.
 \end{proposition}

\begin{proof}
Clearly, $\sigma_x=\sigma_y$ if and only if $\sigma^{-1}_x=\sigma^{-1}_y$ and the same is also true for $\hat\sigma_x$ and $\hat\sigma^{-1}_x$.

Let $a\sim a'$ and $b\sim b'$. 
We know,
by Proposition~\ref{cong-qmagma}, that $\sigma_{\sigma_a^{-1}(b)}^{-1}=\sigma_{\sigma_{a'}^{-1}(b')}^{-1}$
and
$\hat\sigma_{\sigma_a^{-1}(b)}^{-1}=\hat\sigma_{\sigma_{a'}^{-1}(b')}^{-1}$. 
We prove, by an induction on~$k$, that $\sigma_{\sigma_a^{-k}(b)}^{-1}=\sigma_{\sigma_{a'}^{-k}(b')}^{-1}$
and $\hat\sigma_{\sigma_a^{-k}(b)}^{-1}=\hat\sigma_{\sigma_{a'}^{-k}(b')}^{-1}$,
for each $k\geq 0$. The claim is trivial for $k=0$, for $k>0$, we have
\begin{align*}
\sigma_{\sigma_a^{-k}(b)}^{-1}(z)&=
\sigma_a^{-k}(b) \cdot z = (a\cdot \sigma_a^{1-k}(b))\cdot (a\cdot \sigma_a(z))
\stackrel{\eqref{eq:qcs1}}{=}
(\sigma_a^{1-k}(b): a)\cdot(\sigma_a^{1-k}(b)\cdot \sigma_a(z))\\
&\stackrel{\text{ind.}}=(\sigma_{a'}^{1-k}(b'): a)\cdot(\sigma_{a'}^{1-k}(b')\cdot \sigma_{a'}(z))
\stackrel{\eqref{eq:qcs1}}=(a\cdot \sigma_{a'}^{1-k}(b'))\cdot (a\cdot \sigma_{a'}(z))\\
&=(a'\cdot \sigma_{a'}^{1-k}(b'))\cdot (a'\cdot \sigma_{a'}(z))
=\sigma_{a'}^{-k}(b') \cdot z
=\sigma_{\sigma_{a'}^{-k}(b')}^{-1}(z),\\
\hat\sigma_{\sigma_a^{-k}(b)}^{-1}(z)&=
\sigma_a^{-k}(b) : z = (a\cdot \sigma_a^{1-k}(b)): (a\cdot \sigma_a(z))
\stackrel{\eqref{eq:qcs2}}=(\sigma_a^{1-k}(b): a)\cdot(\sigma_a^{1-k}(b): \sigma_a(z))\\
&\stackrel{\text{ind.}}=(\sigma_{a'}^{1-k}(b'): a)\cdot(\sigma_{a'}^{1-k}(b'): \sigma_{a'}(z)
\stackrel{\eqref{eq:qcs2}}=(a'\cdot \sigma_{a'}^{1-k}(b')): (a'\cdot \sigma_{a'}(z))
=\hat\sigma_{\sigma_{a'}^{-k}(b')}^{-1}(z).
\end{align*}
Now let $k$ be the order of~$\sigma_a$. Then
\[\sigma_{\sigma_a(b)}=\sigma_{\sigma^{1-k}_a(b)}=\sigma_{\sigma^{1-k}_{a'}(b')}=\sigma_{\sigma_{a'}(b')}\]
and analogously  $\hat\sigma_{\sigma_a(b)}=\hat\sigma_{\sigma_{a'}(b')}$. The proof for $\sigma_{\hat\sigma_a(b)}$ and $\hat\sigma_{\hat\sigma_a(b)}$ is symmetric 
and therefore $\sim_:$ is compatible with $\sigma$ and $\hat \sigma$.
\end{proof}

We have three formally different equivalences on~$X$, namely $\sim$, $\sim_:$ and $\hat\sim$, all of them being compatible with some operations. It would be fine to examine their compatibility with other definable operations.

\begin{lemma}
    Let $(X,\sigma,\tau)$ be a regular solution and
    suppose that the relation $\sim_:$ is compatible with the operations $\sigma$, $\sigma^{-1}$, $\hat\sigma$ and $\hat\sigma^{-1}$.
    Then $\sim_:$ is compatible with $\tau$ and $\hat\tau$ as well.
    Moreover, if $(X,\sigma,\tau)$ is non-degenerate then $\sim_:$ is compatible with $\tau^{-1}$ if and only if it is compatible with $\hat\tau^{-1}$.
\end{lemma}

\begin{proof}
 Let $a\sim_: a'$ and $b\sim_:b'$. Then $\sigma_b(a)\sim_:\sigma_{b'}(a')$ and
\[\tau_a(b)=\tau_{\sigma^{-1}_b\sigma_b(a)}(b)=
\hat\sigma^{-1}_{\sigma_b(a)}(b)
=\hat\sigma^{-1}_{\sigma_{b'}(a')}(b)
\sim_:\hat\sigma^{-1}_{\sigma_{b'}(a')}(b')
=\tau_{a'}(b')
\]
and analogously for $\hat\tau_a(b)$.

Suppose now that the solution is non-degenerate and that $\sim_:$ is compatible with $\tau^{-1}$. Then
\[\hat\tau^{-1}_b(a)=\sigma_{\tau^{-1}_a(b)}(a)\sim_:\sigma_{\tau^{-1}_{a'}(b')}(a')=\hat\tau^{-1}_{b'}(a')\]
and the other implication is analogous.
\end{proof}

\begin{lemma}
    Let $(X,\sigma,\tau)$ be a regular solution and
    suppose that the relation $\sim$ is compatible with the operations $\sigma^{-1}$ and $\tau$.
    Then $\sim$ is compatible with $\hat\sigma^{-1}$ as well.
\end{lemma}

\begin{proof}
 Let $a\sim a'$ and $b\sim b'$. Then $\sigma_b^{-1}(a)\sim_:\sigma^{-1}_{b'}(a')$ and
\[\hat\sigma^{-1}_a(b)=
\tau_{\sigma_b^{-1}(a)}(b)
=\tau_{\sigma_{b'}^{-1}(a')}(b)
\sim\tau_{\sigma_{b'}^{-1}(a')}(b')
=\hat\sigma^{-1}_{a'}(b'). \qedhere
\]
\end{proof}

It turns out that the equivalences defined above coincide if they are congruences

\begin{proposition}\label{prop:con}
Let $(X,\sigma,\tau)$ be a left non-degenerate solution such that the relation $\sim$ is compatible with the operation $\sigma^{-1}$ and the relation $\sim_:$ is compatible with the operation $\sigma$. Then both the relations are equal.
\end{proposition}

\begin{proof}
Let $x,y,z\in X$. Clearly, $\sigma_x=\sigma_y$ if and only if $\sigma^{-1}_x=\sigma^{-1}_y$.
Suppose now $x\sim y$, that means $\sigma_x=\sigma_y$ and $\tau_x=\tau_y$. Then $\sigma_z^{-1}(x)\sim\sigma_z^{-1}(y)$ and therefore 
$\tau_{\sigma_z^{-1}(x)}(z)=\tau_{\sigma_z^{-1}(y)}(z)$.

On the other hand, if $\sigma^{-1}_x=\sigma^{-1}_y$ and $\tau_{\sigma_z^{-1}(x)}(z)=\tau_{\sigma_z^{-1}(y)}(z)$ then
\[\tau_x(z)=\tau_{\sigma_z^{-1}\sigma_z(x)}(z)
=\tau_{\sigma_z^{-1}\sigma_z(y)}(z)
=\tau_y(z)\]
since $\sigma_z(x)\sim_:\sigma_z(y)$.
\end{proof}

\begin{corollary}
    Let $(X,\sigma,\tau)$ be a finite non-degenerate bijective
    solution. Then the relations $\sim$, $\sim_:$ and $\hat\sim$ coincide.    
\end{corollary}

\begin{proof}
    According to Proposition~\ref{prop:con}, the relations
    $\sim$ and $\sim_:$ are equal. Considering the inverse solution,
    the relations $\hat\sim$ and ${\hat\sim}_:$ are equal.
    But $\sim_:$ and ${\hat\sim}_:$ are equal for any regular solution.
\end{proof}

The following fact was observed by Marco Castelli, although his
proof was different.

\begin{corollary}\label{cor:mutinv}
    Let $(X,\sigma,\tau)$ be a finite non-degenerate bijective
    solution. 
Then the retracts $\mathrm{Ret}(X,\sigma,\tau)$ and $\mathrm{Ret}(X,\hat\sigma,\hat\tau)$ are mutually inverse solutions.

In particular, a finite non-degenerate bijective solution $(X,\sigma,\tau)$ is a multipermutation solution
    of level~$k\geq 0$ if and only if its inverse solution $(X,\hat\sigma,\hat\tau)$
    is multipermutation of~level~$k$.
\end{corollary}

\begin{proof}
    The retracts are equal as sets of equivalence classes.
    Identities \eqref{rr:1}--\eqref{rr:4} then follow from the fact that $\sim$ and $\hat\sim$ are equal congruences.
\end{proof}

\section{Diagonal mappings}\label{sec:diag}

In this section we study diagonal mappings of non-degenerate solutions. If we, for a short moment, assume that a solution $(X,\sigma,\tau)$ is bijective than
we can define another pair of diagonal mappings, namely
$\hat T=\hat\tau_x^{-1}(x)$ and $\hat U=\hat\sigma_x^{-1}(x)$.
Using \eqref{rr:1}--\eqref{rr:4} we have
\begin{align}
	&\hat{T}\colon X\to X; \quad \hat{T}(x)=\sigma_{\tau^{-1}_x(x)}(x)=\sigma_{T(x)}(x),
	\quad {\rm and}\label{eq:barU}\\ 
	&\hat{U}\colon X\to X; \quad \hat{U}(x)=\tau_{\sigma^{-1}_x(x)}(x)=\tau_{U(x)}(x).\label{eq:barT}
\end{align}
Of course, for \eqref{eq:barU} we need right non-degeneracy only and for
\eqref{eq:barT} left non-degeneracy.
\begin{theorem}[Rump {\cite[Corollary 2 of Proposition 8]{Rump19}}]
	Let $(X,\sigma,\tau)$ be a regular solution. Then the following conditions are equivalent:
	\begin{itemize}
		\item the mapping $U$ is bijective,
		\item the mapping $\hat U$ is bijective,
		\item the solution $(X,\sigma,\tau)$ is non-degenerate,
		\item the solution $(X,\hat\sigma,\hat\tau)$ is non-degenerate,
	\end{itemize}
	In such a case, we have $U^{-1}=\hat T$ and $\hat U^{-1}=T$.
\end{theorem}

 Rump's proof uses structural properties of $q$-cycle sets. 
 In this section of our paper, 
 we prove the same result
 without assuming bijectivity,
  using a direct syntactical manipulation only.
  We take  \eqref{eq:barU} and  \eqref{eq:barT} as definitions
  of mappings for any non-degenerate solution (possibly not bijective)
  and we prove $U^{-1}=\hat T$ and $T^{-1}=\hat U$ again.
 Moreover, we also prove that $U$ and $T$ are commuting bijections, not necessarily mutually inverse.

\begin{lemma} 
Let $(X,\sigma,\tau)$ be a non-degenerate solution. Then, for $x\in X$:
\begin{align}
&\sigma_{\hat{U}(x)}=\sigma_{U(x)}\quad {\rm and}\quad \tau_{\hat{T}(x)}=\tau_{T(x)},\label{eq:sigmaS}\\
&\hat{U}(x)=\sigma_{\hat{U}(x)}\hat{U}U(x)\quad {\rm and}\quad \hat{T}(x)=\tau_{\hat{T}(x)}\hat{T}T(x),\label{eq:BS}\\
&\tau_x\hat{T}(x)=\sigma_x\hat{U}(x),
\label{eq:AB} \\
&\hat{T}(x)=\sigma_{\hat{T}(x)}(x)\quad {\rm and}\quad \hat{U}(x)=\tau_{\hat{U}(x)}(x).\label{eq:A}
\end{align}
\end{lemma}

\begin{proof}
For $x,y,z\in X$ we have:\\
\eqref{eq:sigmaS}:
\begin{align*}
&\sigma_x\sigma_y\stackrel{\eqref{birack:1}}=\sigma_{\sigma_x(y)}\sigma_{\tau_y(x)}\quad \Rightarrow\quad \sigma^{-1}_{\sigma_x(y)}\sigma_x=\sigma_{\tau_y(x)}\sigma^{-1}_y\quad \stackrel{y\mapsto \sigma_x^{-1}(x)}\Rightarrow\\
&\id=\sigma_x^{-1}\sigma_x=\sigma_{\tau_{\sigma_x^{-1}(x)}(x)}\sigma^{-1}_{\sigma_x^{-1}(x)}\quad \Rightarrow\quad \sigma_{\hat{U}(x)}\stackrel{\eqref{eq:barU}}=\sigma_{\tau_{\sigma_x^{-1}(x)}(x)}=\sigma_{\sigma_x^{-1}(x)}=\sigma_{U(x)}
\end{align*}
\eqref{eq:BS}: 
\begin{align*}
&\tau_{\sigma_{\tau_y(x)}(z)}\sigma_x(y)\stackrel{\eqref{birack:2}}=\sigma_{\tau_{\sigma_y(z)}(x)}\tau_z(y)\quad \stackrel{z\mapsto \sigma_y^{-1}(z)}\Rightarrow\quad \tau_{\sigma_{\tau_y(x)}\sigma^{-1}_y(z)}\sigma_x(y)=\sigma_{\tau_z(x)}\tau_{\sigma^{-1}_y(z)}(y)\quad \stackrel{y\mapsto \sigma_x^{-1}(y)}\Rightarrow 
\end{align*}
\begin{align}
&\tau_{\sigma_{\tau_{\sigma_x^{-1}(y)}(x)}\sigma^{-1}_{\sigma_x^{-1}(y)}(z)}(y)=
\sigma_{\tau_z(x)}\tau_{\sigma^{-1}_{\sigma^{-1}_x(y)}(z)}\sigma_x^{-1}(y)\quad 
\stackrel{y\mapsto x}\Rightarrow \label{eq:o1}
\end{align}
\begin{align*}
&\tau_z(x)\stackrel{\eqref{eq:sigmaS}}=\tau_{\sigma_{\hat{U}(x)}\sigma^{-1}_{U(x)}(z)}(x)\stackrel{\eqref{eq:barT}}=\tau_{\sigma_{\tau_{U(x)}(x)}\sigma^{-1}_{U(x)}(z)}(x)\stackrel{\eqref{eq:o1}}=
\sigma_{\tau_z(x)}\tau_{\sigma^{-1}_{U(x)}(z)}U(x)\quad \stackrel{z\mapsto U(x)}\Rightarrow
\end{align*}

\begin{align}
&\tau_{U(x)}=
\sigma_{\tau_z(x)}\tau_{\sigma^{-1}_{U(x)}(z)}U(x)\quad \Rightarrow\label{eq:o2}
\end{align}
\begin{align*}
&\hat{U}(x)\stackrel{\eqref{eq:barT}}=\tau_{U(x)}(x)\stackrel{\eqref{eq:o2}}=
\sigma_{\tau_{{U(x)}}(x)}\tau_{\sigma^{-1}_{U(x)}{U(x)}}U(x)\stackrel{\eqref{eq:barT}}=\sigma_{\hat{U}(x)}\hat{U}U(x)
\end{align*}
\eqref{eq:AB}:
\begin{align*}
&\tau_{\sigma_{\tau_y(x)}(z)}\sigma_x(y)\stackrel{\eqref{birack:2}}=\sigma_{\tau_{\sigma_y(z)}(x)}\tau_z(y)\quad \stackrel{x\mapsto \tau_y^{-1}(x)}\Rightarrow\quad
\tau_{\sigma_x(z)}\sigma_{\tau_y^{-1}(x)}(y)=\sigma_{\tau_{\sigma_y(z)}\tau_y^{-1}(x)}\tau_z(y)\quad \stackrel{z\mapsto \sigma_x^{-1}(x)}\Rightarrow\\
&\tau_{x}\sigma_{\tau_y^{-1}(x)}(y)=\sigma_{\tau_{\sigma_y\sigma_x^{-1}(x)}\tau_y^{-1}(x)}\tau_{\sigma^{-1}_x(x)}(y)\quad \stackrel{y\mapsto x}\Rightarrow\quad \tau_{x}\sigma_{\tau_x^{-1}(x)}(x)=\sigma_{x}\tau_{\sigma^{-1}_x(x)}(x)\quad \Rightarrow\\
&\tau_x\hat{T}(x)\stackrel{\eqref{eq:barU}}=\tau_x\sigma_{\tau^{-1}_x(x)}=\sigma_x\tau_{\sigma^{-1}_x(x)}(x)\stackrel{\eqref{eq:barT}}=\sigma_x\hat{U}(x)
\end{align*}
\eqref{eq:A} 
\begin{align*}
&\sigma_x\sigma_y\stackrel{\eqref{birack:1}}=\sigma_{\sigma_x(y)}\sigma_{\tau_y(x)}\quad \stackrel{x\mapsto \tau_y^{-1}(x)}\Rightarrow\quad 
\sigma_{\tau_y^{-1}(x)}\sigma_y=\sigma_{\sigma_{\tau_y^{-1}(x)}(y)}\sigma_{x}\quad\stackrel{y\mapsto x}\Rightarrow\\
&\sigma_{\tau_x^{-1}(x)}\sigma_x=\sigma_{\sigma_{\tau_x^{-1}(x)}(x)}\sigma_x\quad \Rightarrow\quad \hat{T}(x)\stackrel{\eqref{eq:barU}}=\sigma_{\tau_x^{-1}(x)}(x)=\sigma_{\sigma_{\tau_x^{-1}(x)}(x)}(x)\stackrel{\eqref{eq:barU}}=\sigma_{\hat{T}(x)}(x).
\end{align*}
Similarly we can prove the dual equations.
\end{proof}

\begin{theorem}\label{lm:TSbijections}
Let $(X,\sigma,\tau)$ be a non-degenerate solution. The pairs of mappings $\hat{T}$ and $U$ and $\hat{U}$ and $T$ are mutually inverse. 
\end{theorem}
\begin{proof}
At first we will show that $\hat{T}U=U\hat{T}=\id$. 
Let $x\in X$. Then
\begin{align*}
&U\hat{T}(x)=\sigma^{-1}_{\hat{T}(x)}\hat{T}(x)\stackrel{\eqref{eq:A}}=\sigma^{-1}_{\hat{T}(x)}\sigma_{\hat{T}(x)}(x)=x,\quad {\rm and}\\
&\hat{T}U(x)\stackrel{\eqref{eq:AB}}=\tau^{-1}_{U(x)}\sigma_{U(x)}\hat{U}U(x)\stackrel{\eqref{eq:sigmaS}}=\tau^{-1}_{U(x)}\sigma_{\hat{U}(x)}\hat{U}U(x)\stackrel{\eqref{eq:BS}}=\tau^{-1}_{U(x)}\hat{U}(x)\stackrel{\eqref{eq:barT}}=\tau^{-1}_{U(x)}\tau_{U(x)}(x)=x.
\end{align*}
The proof that $\hat{U}T=T\hat{U}=\id$ goes analogously.
\end{proof}

Colazzo, Jespers, Van Antwerpen and Verwimp showed \cite[Lemma 3.3]{CJAV} that, for a non-degenerate solution, if the diagonal operation $U$ is bijective than the solution is bijective. Thus, by Theorem \ref{lm:TSbijections} we immediately obtain a~positive answer to Question 4.3 \cite{CJV} of Ced\'{o}, Jespers and Verwimp.

\begin{corollary}\label{cor:bijective}
Any non-degenerate solution is bijective.
\end{corollary}

If a solution comes from a skew brace, necessarily $U=\hat U$ \cite[Corollary 33]{SV18}. 
Nevertheless, there are many examples where $U$ and $\hat U$ differ.

\begin{example}\label{ex:Lub}
For a non-empty set $X$ and two bijections  $f,g\colon X\to X$ such that $fg=gf$, the  solution $(X,\sigma,\tau)$ with $\sigma_x=f$ and $\tau_{x}=g$, for each $x\in X$, is called \emph{Lyubashenko solution} (see \cite{Dr90}). Clearly, for such solution, we have: $U=f^{-1}$ and $T=g^{-1}$. Hence, if  $g\neq f^{-1}$ then $\hat U\neq U$.  
\end{example}
\begin{example}\label{ex:der}
Let $(X,\sigma,\tau)$ be a derived solution such that $\sigma_x=\id$, for all $x\in X$. In this case $U=\id$ and $\hat U(x)=\tau_x(x)$. For derived solutions with all $\tau_x=\id$, one obtains the dual situation: $T=\id$ and $\hat T(x)=\sigma_x(x)$.
\end{example}

For all the examples that we have, 
the bijections $U$ and $T$ commute. The next theorem shows that it is not a coincidence.

\begin{theorem}
Let $(X,\sigma,\tau)$ be a non-degenerate solution. Then the mappings $U$ and $T$ commute. 
\end{theorem}

\begin{proof}
Let $x\in X$. Then
\begin{align*}
&UT^{-1}U^{-1}(x)=\sigma^{-1}_{T^{-1}U^{-1}(x)}T^{-1}U^{-1}(x)\stackrel{\eqref{eq:sigmaS}}=\\
&\sigma^{-1}_{UU^{-1}(x)}\tau_{UU^{-1}(x)}U^{-1}(x)=
\sigma^{-1}_x\tau_xU^{-1}(x)\stackrel{\eqref{eq:AB}}=
T^{-1}(x).\qedhere
\end{align*}
\end{proof}

The proof turns out to be much easier in the language of $q$-cycle sets:
\begin{theorem}
    Let $(X,\cdot,:)$ be a $q$-cycle set. Then the mappings $U(x)=x\cdot x$ and $\hat U(x)=x:x$ commute. 
\end{theorem}

\begin{proof}
$U(\hat U(x))=(x:x)\cdot(x:x)\stackrel{\eqref{eq:qcs2}}=(x\cdot x):(x\cdot x)=\hat U(U(x))$.
\end{proof}

\begin{example}
    Let $X$ be a set and let $e\in X$. We define, for all $x,y\in X$, to be $x\cdot y=y$ and $x:y=e$. Then $(X,\cdot,:)$ is a $q$-cycle set where $U(x)=x$ and $\hat U(x)=e$, for all~$x\in X$. These mappings commute but only one of them is a permutation.
\end{example}

\section{Multipermutation solutions}\label{sec:multi}
In this section we describe the multipermutation level of a solution using equations.

A solution $(X,\sigma,\tau)$ is called \emph{permutational}, if for every $x,y\in X$, $\sigma_x=\sigma_y$  and $\tau_x=\tau_y$. If it is non-degenerate, it has multipermutation level $1$.
The class of non-degenerate permutational solutions evidently coincides with the class of Lyubashenko solutions.
A permutational solution is a \emph{projection} (or \emph{trivial}) solution if for every $x\in X$, $\sigma_x=\tau_x=\id$.

In \cite[Theorem 5.15]{GIC12} Gateva-Ivanova characterized an involutive multipermutation solution by 
one equation and to describe the equation, she introduced an expression called \emph{a tower of actions} \cite[Definition 5.9]{GIC12}. However, for non-involutive solutions we need much more equations of this type. That is why we decided to introduce a new notation.

Let $X$ be a set and let~$\Gamma$ be a set of mappings $\gamma\colon X^2\to X$, $(x,y)\mapsto \gamma_x(y)$. 
Let us define, for each $i\in \mathbb{N}$, for all $x,z_1,\ldots,z_i\in X$
and for all $\gamma^{(1)},\gamma^{(2)},\ldots,\gamma^{(i-1)},\gamma^{(i)}\in \Gamma$,
the following terms:
\begin{align*}
&\Omega_0(x):=x,\\
&\Omega_1(\gamma^{(1)},x,z_1):=\gamma^{(1)}_x(z_1), \\
&\vdots\\
&\Omega_i(\gamma^{(i)},\gamma^{(i-1)},\ldots,\gamma^{(1)},x,z_1,\ldots,z_{i-1},z_i):=\gamma^{(i)}_{\Omega_{i-1}(\gamma^{(i-1)},\ldots,\gamma^{(1)},x,z_1,\ldots,z_{i-1})}(z_i).
\end{align*}

Note that there are some useful relationships among the above terms.

\begin{lemma}
Let $X$ be a set and let~$\Gamma$ be a set of mappings from $X^2$ to $X$.
Then, for every $m>n\in\mathbb{N}$, $x,y,z_1,\ldots,z_m\in X$ and any $\gamma,\gamma^{(1)},\ldots,\gamma^{(m)}\in \Gamma$:
\begin{align}\label{eq:1}
	\Omega_m(\gamma^{(m)},\ldots,\gamma^{(1)},x,z_1,\ldots,z_m)&=
	\Omega_{m-1}(\gamma^{(m)},\ldots,\gamma^{(2)},\Omega_1(\gamma^{(1)},x,z_1),z_2,\ldots,z_m),\\
	\begin{split}
	\label{eq:4}
	\Omega_m(\gamma^{(m)},\ldots,\gamma^{(1)},x,z_1,\ldots,z_m)&=
	\Omega_{m-n}(\gamma^{(m)},\ldots,\gamma^{(n+1)},\\
	&\strut\kern 1cm \Omega_n(\gamma^{(n)},\ldots,\gamma^{(1)},x,z_1,\ldots,z_n),z_{n+1},\ldots,z_m),
	\end{split}
\end{align}
\end{lemma}

\begin{proof}
\eqref{eq:1}:
We go by an induction on~$m$. Clearly, for $m=1$, we have:
\begin{align*}
	\Omega_1(\gamma^{(1)},x,z_1)=\gamma_x^{(1)}(z_1)=\Omega_0(\Omega_1(\gamma^{(1)},x,z_1)).
\end{align*}
Let us assume, the induction hypothesis is true for some $m\in \mathbb{N}$. Then, by the induction hypothesis, we obtain:
\begin{multline*}
	\Omega_{m+1}(\gamma^{(m+1)},\gamma^{(m)},\ldots,\gamma^{(2)},\gamma^{(1)}, x,z_1,\dots,z_m,z_{m+1})=\gamma^{(m+1)}_{\Omega_m(\gamma^{(m)},\ldots,\gamma^{(1)},x,z_1,\dots,z_m)}(z_{m+1})\stackrel{\text{ind.}}{=}\\
	\gamma^{(m+1)}_{\Omega_{m-1}(\gamma^{(m)},\ldots,\gamma^{(2)},\Omega_1(\gamma^{(1)},x,z_1),z_2,\dots,z_m)}(z_{m+1})=\\
	\Omega_{m}(\gamma^{(m+1)},\gamma^{(m)},\ldots,\gamma^{(2)},\Omega_1(\gamma^{(1)},x,z_1),z_2,\dots,z_m,z_{m+1}).
\end{multline*}
\eqref{eq:4}:
We go by an induction on~$n$. For $n=0$, the claim is true since $\Omega_0(x)=x$. Let us suppose that the induction hypothesis is valid for some $n\geq 0$.
Then 
\begin{multline*}
	\Omega_m(\gamma^{(m)},\ldots,\gamma^{(1)},x,z_1,\ldots,z_m)
	\stackrel{\eqref{eq:1}}{=} \Omega_{m-1}(\gamma^{(m)},\ldots,\gamma^{(2)},\Omega_1(\gamma^{(1)},x,z_1),z_2,\ldots,z_m)\stackrel{\text{ind.}}{=}\\
	\Omega_{m-n-1}(\gamma^{(m)},\ldots,\gamma^{(n+2)}, \Omega_n(\gamma^{(n+1)},\ldots,\gamma^{(2)},\Omega_1(\gamma_1,x,z_1),z_2,\ldots,z_{n+1}),z_{n+2},\ldots,z_m)\stackrel{\eqref{eq:1}}{=}
	\\
	\Omega_{m-n-1}(\gamma^{(m)},\ldots,\gamma^{(n+2)}, \Omega_{n+1}(\gamma^{(n+1)},\ldots,\gamma^{(1)},x,z_1,\ldots,z_{n+1}),z_{n+2},\ldots,z_m) \qedhere
\end{multline*}
\end{proof}

\begin{lemma}\label{lm:u1}
Let $X$ be a set and let~$\Gamma$ be a set of mappings from $X^2$ to $X$ closed on inverses, meaning that, for each $\gamma\in\Gamma$, there exists $\gamma^{-1}\in\Gamma$ such that $\gamma^{-1}_x\gamma_x(y)=y$, for all $x,y\in X$.
 Then, for every $x,y,z_1,\ldots,z_m\in X$ and any $\gamma,\gamma^{(1)},\ldots,\gamma^{(m)}\in \Gamma$:
\begin{align}\label{eq:short1}
\Omega_m(\gamma,\ldots,\gamma,x,\Omega_1(\gamma^{-1},x,x),\dots,&\,\Omega_1(\gamma^{-1},x,x))=x,\\
\label{eq:short2}
\Omega_m(\gamma^{(m-1)},\ldots,\gamma^{(1)},\gamma,x,\Omega_1(\gamma^{-1},x,z_1),z_2,\ldots,z_{m})&=\Omega_{m-1}(\gamma^{(m-1)},\ldots,\gamma^{(1)},z_1,z_2,\dots,z_m).
\end{align}
\end{lemma}
\begin{proof}
The proof will go by an induction with respect to $m$. \\
\eqref{eq:short1}: 
For $m=1$, we have:
\begin{align*}
\Omega_1(\gamma,x,\Omega_1(\gamma^{-1},x,x))=\gamma_x(\Omega_1(\gamma^{-1},x,x))=\gamma_x\gamma_x^{-1}(x)=x.
\end{align*}
Let us assume, the induction hypothesis is true for some $m\in \mathbb{N}$. Then, by the induction hypothesis, we obtain:
\begin{align*}
&\Omega_{m+1}(\gamma,\ldots,\gamma,x,\Omega_1(\gamma^{-1},x,x),\dots,\Omega_1(\gamma^{-1},x,x))=\\
&\gamma_{\Omega_{m}(\gamma,\ldots,\gamma,x,\Omega_1(\gamma^{-1},x,x),\dots,\Omega_1(\gamma^{-1},x,x))}(\Omega_1(\gamma^{-1},x,x))\stackrel{\text{ind.}}{=}\gamma_x\gamma^{-1}_x(x)=x.
\end{align*}
\eqref{eq:short2}: 
Clearly, for $m=1$, we have:
\begin{align*}
\Omega_1(\gamma,x,\Omega_1(\gamma^{-1},x,z_1))=\gamma_x(\Omega_1(\gamma^{-1},x,z_1))=\gamma_x\gamma_x^{-1}(z_1)=z_1=\Omega_0(z_1).
\end{align*}
Let us assume, the induction hypothesis is true, for some $m\in \mathbb{N}$. Then 
we obtain:
\begin{align*}
&\Omega_{m+1}(\gamma^{(m)},\ldots,\gamma^{(1)},\gamma,x,\Omega_1(\gamma^{-1},x,z_1),z_2,\dots,z_m,z_{m+1})=\\\
&\gamma^{(m)}_{\Omega_m(\gamma^{(m-1)},\ldots,\gamma^{(1)},\gamma,x,\Omega_1(\gamma^{-1},x,z_1),z_2,\dots,z_m)}(z_{m+1})\stackrel{\text{ind.}}{=}\\
&\gamma^{(m)}_{\Omega_{m-1}(\gamma^{(m-1)},\ldots,\gamma^{(1)},z_1,z_2,\dots,z_m)}(z_{m+1})=
\Omega_{m}(\gamma^{(m)},\ldots,\gamma^{(1)},z_1,z_2,\dots,z_m,z_{m+1}).
\qedhere
\end{align*}
\end{proof}

The equational characterization by Gateva-Ivanova sounds, in our terms, like this:

\begin{proposition}\label{prop:per}\cite[Proposition 4.7]{GI18}
A non-degenerate involutive solution $(X,\sigma,\tau)$ with at least two elements is a multipermutation solution of level at most $1\leq k$ if and only if the solution satisfies, for all $x,y,z_1,\ldots,z_k\in X$,
the following condition:
\begin{equation}\label{eq:multi}
\Omega_k(\sigma,\ldots,\sigma, x,z_1,\ldots,z_k)=\Omega_k(\sigma,\ldots,\sigma, y,z_1,\ldots,z_k).
\end{equation}
\end{proposition}

Similar characterization exists for non-involutive solutions as well but Equation~\eqref{eq:multi}  is replaced by $4^k$ or $2^k$ ones.

\begin{theorem}\label{thm:multieq}
Let $(X,\sigma,\tau)$ be a non-degenerate solution. The following conditions are equivalent:
\begin{enumerate}
\item [(i)] 
$(X,\sigma,\tau)$ is a multipermutation solution of level at most $k\geq 0$;
\item [(ii)] the equations: 
\begin{equation}\label{eq:kper}
\Omega_k(\gamma^{(k)},\ldots,\gamma^{(1)},x,z_1,\ldots,z_k)=\Omega_k(\gamma^{(k)},\ldots,\gamma^{(1)},y,z_1,\ldots,z_k),
\end{equation}
are satisfied for all $x,y,z_1,\ldots,z_k\in X$ and any $k$-tuples $(\gamma^{(1)},\ldots,\gamma^{(k)})\in \{\sigma,\sigma^{-1},\tau,\tau^{-1}\}^k$;
\item [(iii)] the equations: 
\begin{equation}\label{eq:kper1}
\Omega_k(\gamma^{(k)},\ldots,\gamma^{(1)},x,z_1,\ldots,z_k)=\Omega_k(\gamma^{(k)},\ldots,\gamma^{(1)},y,z_1,\ldots,z_k),
\end{equation}
are satisfied for all $x,y,z_1,\ldots,z_k\in X$ and any $k$-tuples $(\gamma^{(1)},\ldots,\gamma^{(k)})\in \{\sigma,\tau\}^k$.
\end{enumerate}

\end{theorem}
\begin{proof}
The proof goes by induction with respect on $k$. \\
$(i)\Leftrightarrow (ii)$:\qquad
Clearly, a solution is of multipermutation level $0$ if and only if $|X|=1$ or equivalently, for all $x,y\in X$: $$\Omega_0(x)=x=y=\Omega_0(y).$$

Now let us assume, the induction hypothesis is true for some $k\in \mathbb{N}$. Furthermore, a solution is of multipermutation level at most $k+1$ if and only if the solution ${\rm Ret}(X,\sigma,\tau)$ is of multipermutation level at most $k$. By the induction hypothesis, it simply means that 
\begin{align*}
&\Omega_k(\gamma^{(k)},\ldots,\gamma^{(1)},x,z_1,\ldots,z_k)/_\approx=\Omega_k(\gamma^{(k)},\ldots,\gamma^{(1)},x/_\approx,z_1/_\approx,\ldots, z_k/_\approx)=\\
&\Omega_k(\gamma^{(k)},\ldots,\gamma^{(1)},y/_\approx,z_1/_\approx,\ldots,z_k/_\approx)=\Omega_k(\gamma^{(k)},\ldots,\gamma^{(1)},y,z_1,\ldots,z_k)/_\approx
\end{align*}
holds for all $x/_\approx,y/_\approx,z_1/_\approx,\ldots,z_k/_\approx\in {\rm Ret}(X)$ and any $k$-tuples $(\gamma^{(1)},\ldots,\gamma^{(k)})\in \{\sigma,\sigma^{-1},\tau,\tau^{-1}\}^k$. Hence, for all $x,y,z_1,\ldots,z_k\in X$ and any $k$-tuples $(\gamma^{(1)},\ldots,\gamma^{(k)})\in \{\sigma,\sigma^{-1},\tau,\tau^{-1}\}^k$ we obtain the following 
equality in the solution $(X, \sigma,\tau)$:
\begin{equation*}
\Omega_k(\gamma^{(k)},\ldots,\gamma^{(1)},x,z_1,\ldots,z_k)\approx \Omega_k(\gamma^{(k)},\ldots,\gamma^{(1)},y,z_1,\ldots,z_k).
\end{equation*}
This immediately implies that, for all $x,y,z_1,\ldots,z_k,z\in X$ and any $(k+1)$-tuples $(\gamma^{(1)},\ldots,\gamma^{(k)},\gamma)\in \{\sigma,\sigma^{-1},\tau,\tau^{-1}\}^{k+1}$, having multipermutation level at most~$k$ is equivalent to
\begin{align*}
&\Omega_{k+1}(\gamma,\gamma^{(k)},\ldots,\gamma^{(1)},x,z_1,\ldots,z_k,z)=\gamma_{\Omega_k(\gamma^{(k)},\ldots,\gamma^{(1)},x,z_1,\ldots,z_k)}(z)=\\
&\gamma_{\Omega_k(\gamma^{(k)},\ldots,\gamma^{(1)},y,z_1,\ldots,z_k)}(z)=\Omega_{k+1}(\gamma,\gamma^{(k)},\ldots,\gamma^{(1)},y,z_1,\ldots,z_k,z)
\end{align*}
since $a\approx b$ if and only if $\gamma_a=\gamma_b$, for all $\gamma\in\{\sigma,\sigma^{-1},\tau,\tau^{-1}\}$.

$(i)\Leftrightarrow (iii)$:\qquad The proof is nearly the same as for $(i)\Leftrightarrow(ii)$. The only difference comes at the last line where we now claim $a\approx b$ if and only if $\gamma_a=\gamma_b$, for all $\gamma\in\{\sigma,\tau\}$.
\end{proof}

\begin{corollary}\label{cor:mpl_ss}
    Let $(X,\sigma,\tau)$ be a finite regular solution. Then $(X,\sigma,\tau)$ is a multipermutation solution
    of level at most~$k\geq 0$ if and only if
the equations: 
\begin{equation}\label{eq:kper2}
\Omega_k(\gamma^{(k)},\ldots,\gamma^{(1)},x,z_1,\ldots,z_k)=\Omega_k(\gamma^{(k)},\ldots,\gamma^{(1)},y,z_1,\ldots,z_k),
\end{equation}
are satisfied for all $x,y,z_1,\ldots,z_k\in X$ and any $k$-tuples $(\gamma^{(1)},\ldots,\gamma^{(k)})\in \{\sigma^{-1},\hat\sigma^{-1}\}^k$.
\end{corollary}

\begin{proof}
 The proof is analogous to the proof of Theorem \ref{thm:multieq}.
\end{proof}

Theorem \ref{thm:multieq} justifies the following definitions:

\begin{de}\label{de:kper}
Let $k\in \mathbb{N}$. A solution $(X,\sigma, \tau)$ is called $k$-\emph{permutational} if it satisfies \eqref{eq:kper1}.
\end{de}

A non-degenerate solution is a multipermutation solution of~level~$k\geq 1$ if and only if it is $k$-permutational and not $k-1$ permutational, according to Theorem~\ref{thm:multieq}. Nevertheless, there exist permutational solutions that are degenerate.

\begin{exm}
Let $(X,\sigma,\tau)$ be a non-degenerate solution of multipermutation level~$k$. Let $Y=X\times\{0,1\}$ and let, for $x,y\in X$ and $i,j\in \{0,1\}$
\[ \lambda_{(x,i)}((y,j))=(\sigma_{x}(y),j)\quad\text{ and }\quad
\rho_{(x,i)}((y,j))=(\tau_{x}(y),0).\]
Then $(Y,\lambda,\rho)$ is a $k$-permutational left non-degenerate solution.
\end{exm}
	
\begin{exm}
The solution $(X,\sigma,\tau)$ from Example~\ref{exm:1} is not $k$-permutational, for any $k\in\mathbb{N}$, since
$\Omega_k(\sigma,\ldots,\sigma,x,a,\ldots,a)=x$, for both $x\in\{b,c\}$.
\end{exm}
\begin{corollary}\label{lm:u5}
Let $(X,\sigma,\tau)$ be a $k$-permutational solution. Then, for every $x,y,z_1,\ldots,z_{k+1}\in X$ and any $(k+1)$-tuples $(\gamma^{(1)},\ldots,\gamma^{(k+1)})\in\{\sigma,\tau\}^{k+1}$:
\begin{align}\label{eq:short3}
\Omega_{k+1}(\gamma^{(k+1)},\ldots,\gamma^{(1)},x,z_1,\dots,z_k,z_{k+1})=\Omega_{k}(\gamma^{(k+1)},\ldots,\gamma^{(2)},y,z_2,\dots,z_k,z_{k+1}).
\end{align}
\end{corollary}

\begin{proof}
	Use \eqref{eq:kper1} to replace $y$ by $\Omega_1(\gamma^{(1)},x,z_1)$ and then \eqref{eq:1}.
\end{proof}

By Theorem \ref{lm:TSbijections}, we have proved that, in the case of non-degenerate solutions, the mappings~$U$ and $T$ are invertible and the mappings $U^{-1}$ and $T^{-1}$ may be presented as
\begin{align*}
	&U^{-1}(x)= \sigma_{\tau^{-1}_x(x)}(x)=\Omega_2(\sigma,\tau^{-1},x,x,x), 
	\quad {\rm and}\quad \\
	&T^{-1}(x)= \tau_{\sigma^{-1}_x(x)}(x)=\Omega_2(\tau,\sigma^{-1},x,x,x).
\end{align*}
Now we show that, in the case of
$k$-permutational solutions, only one-sided non-degeneracy suffice\cz{s} to prove
	that one of the diagonal mappings is invertible.

\begin{proposition}\label{lm:bijections}
Let $(X,\sigma,\tau)$ be a left non-degenerate $k$-permutational solution. 
Then   
\begin{align*}
&U^{-1}(x)= \Omega_{k}(\sigma,\ldots,\sigma,x,\ldots,x).
\end{align*}
\end{proposition}
\begin{proof}
 Recall,  for $x\in X$, 
$U(x)= \sigma^{-1}_x(x)=\Omega_1(\sigma^{-1},x,x)$.
Hence:
\begin{align*}
&
\Omega_{k}(\sigma,\ldots,\sigma,U(x),\ldots,U(x))=
\Omega_{k}(\sigma,\ldots,\sigma,\Omega_1(\sigma^{-1},x,x),\ldots,\Omega_1(\sigma^{-1},x,x))\stackrel{\eqref{eq:kper}}=\\
&\Omega_{k}(\sigma,\ldots,\sigma,x,\Omega_1(\sigma^{-1},x,x),\ldots,\Omega_1(\sigma^{-1},x,x))\stackrel{\eqref{eq:short1}}=x.
\end{align*}
Further, 
\begin{align*}
&U(\Omega_k(\sigma,\ldots,\sigma,x,\ldots,x))=\Omega_1(\sigma^{-1},\Omega_k(\sigma,\ldots,\sigma,x,\ldots,x),\Omega_k(\sigma,\ldots,\sigma,x,\ldots,x))=\\
&\sigma^{-1}_{\Omega_k(\sigma,\ldots,\sigma,x,\ldots,x)}(\Omega_k(\sigma,\ldots,\sigma,x,\ldots,x))=
\Omega_{k+1}(\sigma^{-1},\sigma,\ldots,\sigma,x,\ldots,x,\Omega_k(\sigma,\ldots,\sigma,x,\ldots,x))
\stackrel{\eqref{eq:short3}}=\\
&\Omega_k(\sigma^{-1},\sigma,\ldots,\sigma,x\ldots,x,\Omega_k(\sigma,\ldots,\sigma,x,\ldots,x))=\\
&
\sigma^{-1}_{\Omega_{k-1}(\sigma,\ldots,\sigma,x,\ldots,x)}(\Omega_k(\sigma,\ldots,\sigma,x,\ldots,x))=\sigma^{-1}_{\Omega_{k-1}(\sigma,\ldots,\sigma, x,\ldots,x)}\sigma_{\Omega_{k-1}(\sigma,\ldots,\sigma, x,\ldots,x)}(x)=x. \qedhere
\end{align*}
\end{proof}

Analogously we can prove that, for a right non-degenerate $k$-permutational solution, we have $T^{-1}(x)= \Omega_{k}(\tau,\ldots,\tau,x,\ldots,x)$.

\begin{example}
    Let $(X,\cdot,1)$ be a group and let $\sigma_x(y)=x\cdot y$ and $\tau_y(x)=1$, for all $x,y\in X$. Then $(X,\sigma,\tau)$ is a left non-degenerate solution that is not $k$-permutational, for any $k\geq 1$, since $\Omega_k(\sigma,\ldots,\sigma,a,1,\ldots,1)=a$, for any $a\in X$. The mapping $U$ is not a permutation since $U(x)=1$, for any $x\in X$.
\end{example}

\section{Reductive solutions}\label{sec:kred}
In this section we study a set of identities similar to the ones that describe the multipermutation level.

\begin{de}\label{de:kred}
Let $k\in \mathbb{N}$. A solution $(X,\sigma, \tau)$ is called $k$-\emph{reductive} if, for every $x,y,z_1,\ldots,z_k\in X$ and any $k$-tuples $(\gamma^{(1)}, \ldots,\gamma^{(k)})\in \{\sigma,\tau\}^k$:
\begin{equation}\label{eq:kred}
\Omega_k(\gamma^{(k)},\ldots,\gamma^{(1)},x,z_1,\ldots,z_k)=\Omega_{k-1}(\gamma^{(k)},\ldots,\gamma^{(2)},z_1,\ldots,z_k).
\end{equation}
\end{de}

In the case of involutive solutions~\cite{JPZ20a} the property of $k$-reductivity was 
defined by 
\begin{align*}
\Omega_k(\sigma,\ldots,\sigma,x,z_1,\ldots,z_k)=\Omega_{k-1}(\sigma,\ldots,\sigma,z_1,\ldots,z_k)
\end{align*}
 only. It is nevertheless easy to prove that all the properties
\eqref{eq:kred} are equivalent for involutive solutions.

Clearly, $1$-reductive solution is square-free since $\gamma_x(x)=\Omega_1(\gamma,x,x)=\Omega_0(x)=x$, for $\gamma\in \{\sigma,\tau\}$. Moreover, each $k$-reductive solution is 
$k$-permutational. 
On the other hand, by Corollary~\ref{lm:u5}, every $k$-permutational solution is $k+1$-reductive.
We can thus say that $k$-reductivity is sort of a $k-\frac 12$-permutability. Nevertheless, there are cases where the degree of
reductivity and permutability agree. 
For instance, the authors together with Zamojska-Dzienio  proved in \cite[Theorem 4.6]{JPZ20b} that if a solution $(X,\sigma,\tau)$ is distributive (i.e. $\sigma_y\sigma_x=\sigma_{\sigma_y(x)}\sigma_y$ and $\tau_y\tau_x=\tau_{\tau_y(x)}\tau_y$ for all $x,y\in X$) then for $k\geq 2$, a solution is $k$-reductive if and only if it is $k$-permutational. It is also true that if a $k$-permutational solution~$(X,\sigma,\tau)$ is square free then, for every $x,y,z_2,\ldots,z_k\in X$ and any $k$-tuples $(\gamma^{(1)},\ldots,\gamma^{(k)})\in \{\sigma,\tau\}^k$:
\begin{align*}
&\Omega_k(\gamma^{(k)},\ldots,\gamma^{(2)},\gamma^{(1)},y,x,z_2,\ldots,z_k)\stackrel{\eqref{eq:kper}}=
\Omega_k(\gamma^{(k)},\ldots,\gamma^{(2)},\gamma^{(1)},x,x,z_2,\ldots,z_k)\stackrel{\eqref{eq:1}}=\\
&\Omega_{k-1}(\gamma^{(k)},\ldots,\gamma^{(2)},\Omega_1(\gamma^{(1)},x,x),z_2,\ldots,z_k)=
\Omega_{k-1}(\gamma^{(k)},\ldots,\gamma^{(2)},x,z_2,\ldots,z_k).
\end{align*}
Hence, each square free non-degenerate solution of multipermutation level~$k$ is $k$-reductive.

The condition of square-freeness can be weakened.
For an involutive solution $(X,\sigma,\tau)$,  Gateva-Ivanova considered in \cite[Definition 4.3]{GI18} a condition saying
\begin{align*}
\tag{$\ast$} \forall x\in X\quad \exists y\in X\quad  \sigma_y(x)=x.
\end{align*}
It is evident, that each square free solution satisfies Condition $(\ast)$. On the other hand solutions without fixed points are examples of ones which do not satisfy this condition.

\begin{fact}\cite[Proposition 8.2]{GIC12}, \cite[Proposition 4.7]{GI18}
If a non-degenerate involutive solution satisfies Condition~$(\ast)$ then it is $k$-permutational if and only if it is $k$-reductive.
\end{fact}
For non-involutive solutions we have a similar result.

\begin{proposition}\label{prop:star}
Let $k\in \mathbb{N}$.  Let $(X,\sigma,\tau)$ be a solution satisfying the following two properties:
  \begin{align}
    \forall x\in X \ \exists y_x\in X\quad \sigma_{y_x}(x)&=\Omega_1(\sigma,y_x,x)=x,\label{prop:star:1}\\
    \forall x\in X \ \exists z_x\in X\quad \tau_{z_x}(x)&=\Omega_1(\tau,z_x,x)=x.\label{prop:star:2}
  \end{align}
  Then $(X,\sigma,\tau)$ is $k$-permutational if and only if it is $k$-reductive.
\end{proposition}

\begin{proof}
Let $(X,\sigma,\tau)$ be $k$-permutational solution which satisfies \eqref{prop:star:1}--\eqref{prop:star:2}. Then for each $z\in X$ there exist $a_z,b_z\in X$ such that 
\[
\sigma_{a_z}(z)=z\quad {\rm and}\quad \tau_{b_z}(z)=z.
\]
Hence for every $x,z,z_2,\ldots,z_k\in X$,  any $k$-tuples $(\gamma^{(1)},\ldots,\gamma^{(k)})\in \{\sigma,\tau\}^k$ and suitable $a_z$ and $b_z$ we have:
\begin{align*}
&\Omega_k(\gamma^{(k)},\ldots,\gamma^{(2)},\sigma,x,z,z_2,\ldots,z_k)\stackrel{\eqref{eq:kper1}}=\Omega_k(\gamma^{(k)},\ldots,\gamma^{(2)},\sigma, a_z,z,z_2,\ldots,z_k)\stackrel{\eqref{eq:1}}=\\
&\Omega_{k-1}(\gamma^{(k)},\ldots,\gamma^{(2)},\Omega_1(\sigma,a_z,z),z_2,\ldots,z_k)\stackrel{\eqref{prop:star:1}}=
\Omega_{k-1}(\gamma^{(k)},\ldots,\gamma^{(2)},z,z_2,\ldots,z_k),
\end{align*}
and analogously for $\tau$,  which completes the proof.
\end{proof}

In the case of a non-degenerate involutive solution $(X,\sigma,\tau)$, the permutation group is defined as the subgroup $\Mlt(X)=\langle\sigma_x:x\in X\rangle$ of the symmetric group $S(X)$ generated by all translations $\sigma_x$, with $x\in X$. Orbits of the action $\Mlt(X)$ on $X$ are referred to as the orbits of a solution $(X,\sigma,\tau)$. A solution $(X,\sigma,\tau)$ is \emph{indecomposable} if the permutation group  $\mathcal{G}(X)$ acts transitively on $X$. 
In the general non-degenerate case, the permutation group of a solution is more complicated. 
Bachiller defined in \cite[Definition 3.10]{B18} the permutation group of the solution as some subgroup of the product $S(X)\times S(X)$. Ced\'o et al. showed in \cite[Lemma 1.3]{CJKAV} that such group is isomorphic to the group generated by all pairs of the form $(\sigma_x,\tau^{-1}_x)$, with $x\in X$. Following Stefanello and Trappeniers \cite{ST} we will treat orbits of a solution as a subsets of $X$ closed under all maps $\sigma_x, \sigma^{-1}_x,\tau_x,\tau_x^{-1}$, for $x\in X$.

Rump proved in \cite{Rump05} that every finite non-degenerate  square-free involutive solution is decomposable and he gave an example of an indecomposable infinite square-free involutive solution. Gateva-Ivanova and Cameron showed in \cite{GIC12} that non-degenerate square-free involutive solutions (of arbitrary cardinality) with finite multipermutation level are decomposable. We will show that each non-degenerate square-free solution 
of multipermutation level $k$ and 
arbitrary cardinality is always decomposable.

Gateva-Ivanova showed in \cite[Theorem 4.14]{GIC12} that, for a non-degenerate involutive multipermutation square-free solution $(X,\sigma,\tau)$ of level $k$, its orbits are multipermutation solutions of level $k-1$. Similar result is also true for non-involutive 
 solutions. 
\begin{theorem}\label{thm:orb}
Orbits of non-degenerate $k$-reductive solutions are $(k-1)$-permutational. 
\end{theorem}

\begin{proof}
We first observe, by substituting $z_1\mapsto\Omega_1\big(\big(\gamma^{(1)}\big)^{-1},x,z_1\big)$ in~\eqref{eq:kred} that a
$k$-reductive solution satisfies, for all $\gamma^{(1)},\ldots,\gamma^{(k)}\in\{\sigma,\tau\}$ and $x,z_1,\ldots,z_k$,
\begin{equation}\label{eq:kred2}
	\Omega_k(\gamma^{(k)},\ldots,\gamma^{(2)},\big(\gamma^{(1)}\big)^{-1},x,z_1,\ldots,z_k)=\Omega_{k-1}(\gamma^{(k)},\ldots,\gamma^{(2)},z_1,\ldots,z_k),
\end{equation}
since
\begin{align*}
&\Omega_k(\gamma^{(k)},\ldots,\gamma^{(2)},\big(\gamma^{(1)}\big)^{-1},x,z_1,\ldots,z_k)\stackrel{\eqref{eq:1}}=\Omega_{k-1}(\gamma^{(k)},\ldots,\gamma^{(2)},\Omega_1(\big(\gamma^{(1)}\big)^{-1},x,z_1),z_2,\ldots,z_k)\stackrel{\eqref{eq:kred}}=\\
&\Omega_{k}(\gamma^{(k)},\ldots,\gamma^{(2)}\gamma^{(1)},x,\Omega_1(\big(\gamma^{(1)}\big)^{-1},x,z_1),z_2,\ldots,z_k)\stackrel{\eqref{eq:short2}}=\Omega_{k-1}(\gamma^{(k)},\ldots,\gamma^{(2)},z_1,\ldots,z_k).
\end{align*}

Now, let $(X,\sigma,\tau)$ be a $k$-reductive solution and $Orb(a)=\{\varphi(a)\colon \varphi\in \langle\sigma_x,\tau_x:x\in X\rangle\}$ be the orbit of $a\in X$. 
Let $x,y\in Orb(a)$. 
Since $x$ and $y$ lie in the same orbit, there exist $n\in\mathbb{N}$,  $\delta^{(1)},\ldots,\delta^{(n)}\in \{\sigma,\sigma^{-1},\tau,\tau^{-1}\}$ and
$c_1,\ldots,c_n\in X$ such that $x=\delta^{(1)}_{c_1}\cdots\delta^{(n)}_{c_n}(y)$.
Hence, for any $\gamma^{(k-1)},\ldots,\gamma^{(1)}\in \{\sigma,\tau\}$
and $z_1,\ldots,z_{k-1}\in Orb(a)$:
\begin{multline*}
\Omega_{k-1}(\gamma^{(k-1)},\ldots,\gamma^{(1)},x,z_1,\ldots,z_{k-1})=\\
\Omega_{k-1}(\gamma^{(k-1)},\ldots,\gamma^{(1)},\Omega_1(\delta^{(1)},c_1,\delta^{(2)}_{c_2}\cdots\delta^{(n)}_{c_n}(y)),z_1,\ldots,z_{k-1})\stackrel{\eqref{eq:1}}=\\
\Omega_{k}(\gamma^{(k-1)},\ldots,\gamma^{(1)},\delta^{(1)},c_1,\delta^{(2)}_{c_2}\cdots\delta^{(n)}_{c_n}(y),z_1,\ldots,z_{k-1})
\stackrel{\eqref{eq:kred}\text{ or }\eqref{eq:kred2}}=\\
\Omega_{k-1}(\gamma^{(k-1)},\ldots,\gamma^{(1)},\delta^{(2)}_{c_2}\cdots\delta^{(n)}_{c_n}(y),z_1,\ldots,z_{k-1})
\stackrel{\text{inductively}}=
\Omega_{k-1}(\gamma^{(k-1)},\ldots,\gamma^{(1)},y,z_1,\ldots,z_{k-1})
\end{multline*}
which means that each orbit $Orb(a)$ is $(k-1)$-permutational,
according to Theorem~\ref{thm:multieq}.
\end{proof}
Directly by Theorem \ref{thm:orb}, we obtain that indecomposable  multipermutation solution of level~$k$ cannot be $k$-reductive.

\begin{corollary}[see {\cite[Proposition 3.3]{CT}}]\label{cor:k-red-dec}
Each non-degenerate $k$-reductive solution of multipermutation level~$k$ is decomposable.
\end{corollary}

\begin{corollary}[see {\cite[Proposition 3.5]{CT}}] Non-degenerate multipermutation  solutions satisfying \eqref{prop:star:1}--\eqref{prop:star:2} are decomposable.
\end{corollary}

\begin{corollary}\label{cor:sqf-dec}
Non-degenerate multipermutation square-free solutions are decomposable.
\end{corollary}

Analogously as in the previous section, there is a structural description what reductivity means.
In \cite{CT} Castelli and Trappeniers presented a variant of the definition of the multipermutation level of a solution.

\begin{de}\cite[Definition 3.1]{CT}
Let $(X,\sigma,\tau)$ be a non-degenerate  multipermutation solution. $mpl'(X,\sigma,\tau)$ is the smallest $k\in \mathbb{N}$ such that ${\rm Ret}^{k}(X,\sigma,\tau)$ is a trivial solution (possibly of size greater than $1$).
\end{de}

We shall connect this notion with our notion of reductivity.

\begin{lemma}\label{lm:retred}
	A non-degenerate solution $(X,\sigma,\tau)$ is $k$-reductive if and only if ${\rm Ret}(X,\sigma,\tau)$ is $(k-1)$-reductive.

\end{lemma}

\begin{proof}
		For $x,z_1,\ldots,z_k\in X$ and any $k$-tuples $(\gamma^{(1)},\ldots,\gamma^{(k)})\in \{\sigma,\sigma^{-1},\tau,\tau^{-1}\}^k$ we have:
	\begin{align*}
		&\Omega_k(\gamma^{(k)},\ldots,\gamma^{(1)},x,z_1,\ldots,z_k)=
		\Omega_{k-1}(\gamma^{(k)},\ldots,\gamma^{(2)},z_1,\ldots,z_k)\quad \Leftrightarrow\\
		&\gamma^{(k)}_{\Omega_{k-1}(\gamma^{(k-1)},\ldots,\gamma^{(1)},x,z_1,\ldots,z_{k-1})}(z_k)=\gamma^{(k)}_{\Omega_{k-2}(\gamma^{(k-1)},\ldots,\gamma^{(2)},z_1,\ldots,z_{k-1})}(z_{k})\quad \Leftrightarrow\\
		&\Omega_{k-1}(\gamma^{(k-1)},\ldots,\gamma^{(1)},x,z_1,\ldots,z_{k-1})\approx
		\Omega_{k-2}(\gamma^{(k-1)},\ldots,\gamma^{(2)},z_1,\ldots,z_{k-1})\quad \Leftrightarrow\\
		&\Omega_{k-1}(\gamma^{(k-1)},\ldots,\gamma^{(1)},x,z_1,\ldots,z_{k-1})/_\approx=
		\Omega_{k-2}(\gamma^{(k-1)},\ldots,\gamma^{(2)},z_1,\ldots,z_{k-1})/_\approx\quad \Leftrightarrow\\
		&\Omega_{k-1}(\gamma^{(k-1)},\ldots,\gamma^{(1)},x/_\approx,z_1/_\approx,\ldots,z_{k-1}/_\approx)=
		\Omega_{k-2}(\gamma^{(k-1)},\ldots,\gamma^{(2)},z_1/_\approx,\ldots,z_{k-1}/_\approx).
	\end{align*}
\end{proof}

\begin{corollary}\label{cor:redprim}
	Let $(X,\sigma,\tau)$ be a non-degenerate  $k$-reductive solution. Then ${\rm Ret}^{k-1}(X,\sigma,\tau)$ is a trivial solution.
\end{corollary}
\begin{proof}
	By Lemma \ref{lm:retred},  ${\rm Ret}^{k-1}(X,\sigma,\tau)$ is a $1$-reductive solution, that means a square-free permutation solution
	and, in consequence, trivial.
\end{proof}

Clearly, for a non-degenerate solution $(X,\sigma,\tau)$ of multipermutation level $k$, 
\begin{align*}
&mpl'(X,\sigma,\tau)\leq k\leq mpl'(X,\sigma,\tau)+1.
\end{align*}
In \cite[Proposition 3.3]{CT} the authors showed that, if a solution of multipermutation level $k$ is indecomposable, then $k=mlp'(X,\sigma,\tau)$. It is actually the same result as our Corollary~\ref{cor:k-red-dec}.
Moreover, if a solution of multipermutation level $k$ satisfies the conditions \eqref{prop:star:1}--\eqref{prop:star:2} then $k=mpl'(X,\sigma,\tau)+1$. 
In other words, reductivity and $mpl'$ measure the same property of a solution and,
by Corollary \ref{cor:redprim}, we obtain:

\begin{corollary}
Let $(X,\sigma,\tau)$ be a non-degenerate $k$-reductive solution which is not $k-1$-reductive. Then 
\begin{align*}
mpl'(X,\sigma,\tau)=k-1.
\end{align*}
\end{corollary}

To finish our paper with the diagonal mappings again, we shall observe that,
by Proposition \ref{lm:bijections}, in the case of left non-degenerate $k$-reductive solution we immediately obtain, using~\eqref{eq:kred}, a shorter form of the mappings $U^{-1}$. 
\begin{corollary}
	Let $(X,\sigma,\tau)$ be a left non-degenerate $k$-reductive solution. Then   
		$U^{-1}(x)= \Omega_{k-1}(\sigma,\ldots,\sigma,x,\ldots,x)$. 
\end{corollary}

\end{document}